\documentclass[a4,12pt,epsfig]{article}
\usepackage{graphicx}
\usepackage{subfig}
\usepackage{amsmath}
\usepackage{amssymb}
\usepackage{amsfonts}
\usepackage{euscript}
\begin{document}
\bibliographystyle{plain}
\newcommand{\bea}{\begin{eqnarray}}
\newcommand{\eea}{\end{eqnarray}}
\newcommand{\bfmN}{{\mbox{\boldmath{$N$}}}}
\newcommand{\bfmx}{{\mbox{\boldmath{$x$}}}}
\newcommand{\bfmv}{{\mbox{\boldmath{$v$}}}}
\newcommand{\se}{\setcounter{equation}{0}}
\newtheorem{corollary}{Corollary}[section]
\newtheorem{example}{Example}[section]
\newtheorem{definition}{Definition}[section]
\newtheorem{theorem}{Theorem}[section]
\newtheorem{proposition}{Proposition}[section]
\newtheorem{lemma}{Lemma}[section]
\newtheorem{remark}{Remark}[section]
\newtheorem{result}{Result}[section]
\newcommand{\vtwo}{\vskip 4ex}
\newcommand{\vthree}{\vskip 6ex}
\newcommand{\vfour}{\vspace*{8ex}}
\newcommand{\hone}{\mbox{\hspace{1em}}}
\newcommand{\hon}{\mbox{\hspace{1em}}}
\newcommand{\htwo}{\mbox{\hspace{2em}}}
\newcommand{\hthree}{\mbox{\hspace{3em}}}
\newcommand{\hfour}{\mbox{\hspace{4em}}}
\newcommand{\von}{\vskip 1ex}
\newcommand{\vone}{\vskip 2ex}
\newcommand{\n}{\mathfrak{n} }
\newcommand{\m}{\mathfrak{m} }
\newcommand{\q}{\mathfrak{q} }
\newcommand{\aF}{\mathfrak{a} }

\newcommand{\kl}{\mathcal{K}}
\newcommand{\p}{\mathcal{P}}
\newcommand{\Lt}{\mathcal{L}}
\newcommand{\bv}{{\mbox{\boldmath{$v$}}}}
\newcommand{\bc}{{\mbox{\boldmath{$c$}}}}
\newcommand{\bx}{{\mbox{\boldmath{$x$}}}}
\newcommand{\br}{{\mbox{\boldmath{$r$}}}}
\newcommand{\bs}{{\mbox{\boldmath{$s$}}}}
\newcommand{\bb}{{\mbox{\boldmath{$b$}}}}
\newcommand{\ba}{{\mbox{\boldmath{$a$}}}}
\newcommand{\bn}{{\mbox{\boldmath{$n$}}}}
\newcommand{\bp}{{\mbox{\boldmath{$p$}}}}
\newcommand{\by}{{\mbox{\boldmath{$y$}}}}
\newcommand{\bz}{{\mbox{\boldmath{$z$}}}}
\newcommand{\be}{{\mbox{\boldmath{$e$}}}}
\newcommand{\proof}{\noindent {\sc Proof :} \par }
\newcommand{\bP}{{\mbox{\boldmath{$P$}}}}

\newcommand{\M}{\mathcal{M}}
\newcommand{\R}{\mathbb{R}}
\newcommand{\Q}{\mathbb{Q}}
\newcommand{\Z}{\mathbb{Z}}
\newcommand{\N}{\mathbb{N}}
\newcommand{\C}{\mathbb{C}}
\newcommand{\xar}{\longrightarrow}
\newcommand{\ov}{\overline}
 \newcommand{\rt}{\rightarrow}
 \newcommand{\om}{\omega}
 \newcommand{\wh}{\widehat }
 \newcommand{\wt}{\widetilde }
 \newcommand{\g}{\Gamma}
 \newcommand{\lm}{\lambda}

\newcommand{\eN}{\EuScript{N}}
\newcommand{\ncom}{\newcommand}
\newcommand{\norm}{\|\;\;\|}
\newcommand{\inp}[2]{\langle{#1},\,{#2} \rangle}
\newcommand{\nrm}[1]{\parallel {#1} \parallel}
\newcommand{\nrms}[1]{\parallel {#1} \parallel^2}
\title{On the Convergence of Quasilinear Viscous Approximations with Degenerate Viscosity }
\author{ Ramesh Mondal\\ Email: rmondal86@gmail.com\\ Department of Mathematics, University of Kalyani, Kalyani, India-741235.}
\maketitle{}
\begin{abstract}
We use Velocity Averaging lemma to show that the almost everywhere limit of quasilinear viscous approximations is the unique entropy
solution (in the sense of {\it F. Otto}) of the corresponding scalar conservation laws on a bounded domain in $\mathbb{R}^{d}$, where the viscous term is
of the form $\varepsilon\,div\left(B(u^{\varepsilon})\nabla u^{\varepsilon}\right)$ and $B\geq 0$.
 \end{abstract}
 \section{Introduction}
 Let $\Omega$ be a bounded domain in $\mathbb{R}^{d}$ with smooth boundary $\partial \Omega$. For $T >0$, denote $\Omega_{T}:= \Omega\times(0,T)$. 
 We write the initial boundary value problem $\left(\mbox{IBVP}\right)$ for scalar conservation laws given by
 \begin{subequations}\label{ivp.cl}
\begin{eqnarray}
  u_t + \nabla\cdot f(u) =0& \mbox{in }\Omega_T,\label{ivp.cl.a}\\
u(x,t)= 0&\mbox{on}\,\,\partial \Omega\times(0,T),\label{ivp.cl.b}\\
  u(x,0) = u_0(x)& x\in \Omega.\label{ivp.cl.c}
  \end{eqnarray}
\end{subequations}
where $f=(f_{1},f_{2},\cdots,f_{d})$ is the flux function and $u_{0}$ is the initial condition.\\
Consider the IBVP for the generalized viscosity problem 
\begin{subequations}\label{De.regularized.IBVP}
\begin{eqnarray}
 u^\varepsilon_{t} + \nabla \cdot f(u^{\varepsilon}) = \varepsilon\,\nabla\cdot\left(B(u^\varepsilon)\,\nabla u^\varepsilon\right)
 &\mbox{in }\Omega_{T},\label{De.regularized.IBVP.a} \\
    u^\varepsilon(x,t)= 0&\,\,\,\,\mbox{on}\,\, \partial \Omega\times(0,T),\label{De.regularized.IBVP.b}\\
u^{\varepsilon}(x,0) = u_{0\varepsilon}(x)& x\in \Omega,\label{De.regularized.IBVP.c}
\end{eqnarray}
\end{subequations}
indexed by $\varepsilon>0$. The problem of the form \eqref{De.regularized.IBVP.a} is posed in \cite{MR2150387}. The aim of this article is to prove that the {\it a.e.} limit of sequence of solutions $\left(u^{\varepsilon}\right)$ to \eqref{De.regularized.IBVP}(called quasilinear viscous approximations) is the unique entropy solution for IBVP \eqref{ivp.cl} in the sense of Otto \cite{MR1387428}. \\
The equation of the form \eqref{De.regularized.IBVP.a} appears in viscous shallow water problem \cite{GQCHEN} and in the equations of gas dynamics for viscous, heat conducting fluid in eulerian coordinates \cite{JSmoller}. The equation of the form \eqref{De.regularized.IBVP.a} is also very important in the study of numerical analysis. Kurganov and  Liu \cite{AKurganov} proposed a finite volume method by introducing an adaptive way of adding viscosity in the shock region for solving general multidimentional system of conservation laws. Their scheme captures numerically stable solution of hyperbolic conservation laws. The coefficient of the added numerical viscosity appears as a function of the solution and therefore the convergence of the scheme to the entropy solution is similar to the problem considered in this article. Mishra and Spinolo \cite{SMishra} have designed schemes by incorporating the explicit information about the underlying viscous operators for systems in one space dimension. As a consequence, these schemes approximate the viscosity solution of conservation laws.\\
\vspace{0.1cm}\\
We now give Hypothesis on $f$, $u_{0}$, $B$. \\
\noindent{\bf Hypothesis G}
\begin{enumerate}
\item[(a)] Let $f\in \left(C^{2}(\mathbb{R})\right)^d$, $f^\prime\in \left(L^\infty(\mathbb{R})\right)^d$, and denote 
$$\|f^\prime\|_{\left(L^\infty(\mathbb{R})\right)^d}:=\displaystyle\max_{1\leq j\leq d}\left(\sup_{y\in\mathbb{R}}|f^\prime(y)|\right).$$
\item[(b)] Let $B\in C^1(\mathbb{R})\cap L^\infty(\mathbb{R})$, and $B\geq 0$.
\item[(c)] Let $u_{0}$ be in $ L^{\infty}(\Omega)$. There exists a sequence $\left(u_{0\varepsilon}\right)\in\mathcal{D}(\Omega)$ and a constant $A>0$ such that $\|u_{0\varepsilon}\|\leq A$. Denote $I:=[-A, A]$. 
\item[(d)] Let $1<p\leq 2$ and $p^{\prime}>1$ be such that $\frac{1}{p}+\frac{1}{p^{\prime}}=1$. Let $\psi\in L^{p^{\prime}}\left(\mathbb{R}\right)$ having essential compact support.
Let $f$ satisfy the following 
\begin{equation*}
\begin{split}
\mbox{meas}\Big\{c\in\mbox{supp}\,\psi,\,\,\tau+\,\left(f_{1}^{\prime}(c),f_{2}^{\prime}(c),\cdots,f_{d}^{\prime}(c)\right)\cdot\xi=0\Big\}=0\,\, \\
\mbox{for all}\,\left(\tau,\xi\right)\in\mathbb{R}\times\mathbb{R}^{d}\,\, \mbox{with}\,\,\tau^{2} +\left|\xi\right|^{2}=1.
\end{split}
\end{equation*}
\end{enumerate}
We recall \textbf{Hypothesis D} from \cite{Ramesh_Mondal} as we use \textbf{Hypothesis D} in Section \ref{NonDegenerate.Section.1} to conclude many results. \\
\vspace{0.1cm}\\
\noindent{\bf Hypothesis D}
\begin{enumerate}
	\item[(a)] Let $f\in \left(C^{2}(\mathbb{R})\right)^d$, $f^\prime\in \left(L^\infty(\mathbb{R})\right)^d$, and denote 
	$$\|f^\prime\|_{\left(L^\infty(\mathbb{R})\right)^d}:=\displaystyle\max_{1\leq j\leq d}\left(\sup_{y\in\mathbb{R}}|f^\prime(y)|\right).$$
	\item[(b)] Let $B\in C^1(\mathbb{R})\cap L^\infty(\mathbb{R})$, and there exists an $r>0$ such that $B\geq r$.
	\item[(c)] Let $u_{0}$ be in $ L^{\infty}(\Omega)$. There exists a sequence $\left(u_{0\varepsilon}\right)\in\mathcal{D}(\Omega)$ and a constant $A>0$ such that $\|u_{0\varepsilon}\|\leq A$. Denote $I:=[-A, A]$. 
	\item[(d)] Let $1<p\leq 2$ and $p^{\prime}>1$ be such that $\frac{1}{p}+\frac{1}{p^{\prime}}=1$. Let $\psi\in L^{p^{\prime}}\left(\mathbb{R}\right)$ having essential compact support.
	Let $f$ satisfy the following 
	\begin{equation*}
	\begin{split}
	\mbox{meas}\Big\{c\in\mbox{supp}\,\psi,\,\,\tau+\,\left(f_{1}^{\prime}(c),f_{2}^{\prime}(c),\cdots,f_{d}^{\prime}(c)\right)\cdot\xi=0\Big\}=0\,\, \\
	\mbox{for all}\,\left(\tau,\xi\right)\in\mathbb{R}\times\mathbb{R}^{d}\,\, \mbox{with}\,\,\tau^{2} +\left|\xi\right|^{2}=1.
	\end{split}
	\end{equation*}
\end{enumerate}
It is well known that it is incorrect to assume boundary data pointwise for hyperbolic problems posed on bounded domain. This is because any smooth solution is constant on any maximal segment of characteristic and we assume that this segment intersects $\Omega\times\left\{0\right\}$ and $\partial\Omega\times (0,T)$. Then the informations which propagate through this segment from initial datum to the boundary $\partial\Omega\times (0,T)$ may not match with the prescribed boundary condition. In other words IBVP for hyperbolic problems are ill-posed in general \cite{MR542510}.  Thus it is required to give a meaning to the way in which the boundary conditions are realized. In BV-setting, Bardos-LeRoux-Nedelec prove the existence, uniqueness of entropy solutions by introducing a concept of entropy solutions which includes boundary term. This concept of entropy solution is known as BLN entropy condition. Thus BLN entropy condition gives a way to interpret the boundary conditions whenever the entropy solution is of bounded variations. We have proved that the {\it a.e.} limit of quasilinear viscous approximations is the unique entropy solution in the sense of Bardos-Leroux-Nedelec in \cite{Ramesh} by establishing BV-estimates of quasilinear viscous approximations with $B$ as mentioned in \textbf{Hypothesis D} and BV initial data. BLN entropy condition is not meaningful for all those solutions which are not of bounded variations as the BLN condition uses the trace of the solution on the boundary $\partial\Omega\times (0,T)$. Otto \cite{MR1387428} generalizes the BLN entropy condition to the $L^{\infty}-$ setting by introducing a new concept of entropy solutions. In this concept of entropy solutions, Otto introduces ``boundary entropy-entropy flux" pair and ask the boundary condition to hold in an integral form. In the $L^{\infty}-$ setting, Otto also proves the existence, uniqueness of entropy solution satisfying the concept of entropy introduced by him. As a result, Otto's entropy condition gives another way to interpret the boundary conditions. We refer the reader to Martin \cite{Martin}, Carrillo \cite{Carrillo}, Vallet \cite{Vallet} in the $L^{\infty}$-setting, and to Porretta and Vovelle \cite{Porretta}, Ammar et al. \cite{Ammar} in $L^{1}$-setting for extensions of the concepts of entropy solutions to such contexts. Dubois and LeFloch \cite{Dubois_Lefloch} generalizes BLN entropy condition by introducing a boundary entropy inequality in the domain $\left\{(x,t)\,:\,x>0,\,t>0\right\}$ for general viscosity term as considered in this article for system of hyperbolic equations.\\
\vspace{0.1cm}\\
In this context, we have the following result.
\begin{theorem}\label{Deg.paper2.compensatedcompactness.theorem1}
 {\rm Let $f,\,B,\,u_{0}$ satisfy \textbf{Hypothesis G}. Then the {\it a.e.} limit of the quasilinear viscous approximations 
 $\left(u^{\varepsilon}\right)$ is the unique entropy solution of IBVP \eqref{ivp.cl} in the sense of {\it Otto}\cite{MR1387428}.}
\end{theorem}
In \cite{Ramesh_Mondal}, we use the method of Compensated Compactness for space dimension $d=1$ and $d=2$ and the Kinetic formulation for the multidimensional case to show the existence of {\it a.e.} convergent subsequence of the quasilinear viscous approximations with $L^{\infty}\left(\Omega\right)$ initial data and \textbf{Hypothesis D}. In \textbf{Hypothesis D}, we assume $B\geq r>0$. In this work, we use velocity averaging lemma to show the existence of {\it a.e.} convergent subsequence of the quasilinear viscous approximations with $L^{\infty}\left(\Omega\right)$ initial data and \textbf{Hypothesis G}. In \textbf{Hypothesis G}, we assume $B\geq 0$. The main difficulties to prove the existence of an {\it a.e.} convergent subsequence of $\left(u^{\varepsilon}\right)$ with \textbf{Hypothesis G} are the following. 
\begin{enumerate}
	\item The $L^{\infty}\left(\Omega_{T}\right)-$ estimate of the quasilinear viscous approximations $\left(u^{\varepsilon}\right)$.
	\item The quasilinear viscous approximations $\left(u^{\varepsilon}\right)$ satisfies the Kruzhkov entropy dissipation \cite{Lions} in the sense of distributions. 
\end{enumerate}
The $L^{\infty}\left(\Omega_{T}\right)-$ estimates of the quasilinear viscous approximations $\left(u^{\varepsilon}\right)$ is proved in Lemma \ref{Degenerate.Maximum.Principle}. The idea of the proof is that we need to use the $L^{\infty}\left(\Omega_{T}\right)-$ estimates of the sequence of solutions of IBVP \eqref{regularized.IBVP} with \textbf{Hypothesis D}, $r=\delta>0$ and the convergence of  $u^{\varepsilon,\delta}\to u^{\varepsilon}$ as $\delta\to 0$ in $L^{1}\left(\Omega_{T}\right)$. \\
The sequence of solutions $\left(u^{\varepsilon}\right)$ satisfies a Kruzhkov entropy dissipation is proved in Theorem \ref{Degenerate.measure.1}. Let $\left(\eta,q\right)$ be the Kruzhkov entropy-entropy flux pair. Then the Kruzhkov entropy dissipation like property is given by the following.
$$\frac{\partial}{\partial t}\eta(u^{\varepsilon};c) + \displaystyle\sum_{j=1}^{d}\,\frac{\partial}{\partial x_{j}}q_{j}(u^{\varepsilon};c) = \displaystyle\lim_{\delta\to 0}\left(\varepsilon\,\eta^{\prime}\left(u^{\varepsilon,\delta};c\right)\displaystyle\sum_{j=1}^{d}\frac{\partial}{\partial x_{j}}\left(B(u^{\varepsilon,\delta})\,\frac{\partial u^{\varepsilon,\delta}}{\partial x_{j}}\right)\right)$$ 
in $\mathcal{D}^{\prime}\left(\Omega\times (0,T)\right)$. This Kruzhkov's entropy dissipation like property is proved in Theorem \ref{Degenerate.measure.1} by following the computation of Proposition 4.1 of \cite{Ramesh_Mondal} and
passing to the limit as $\delta\to 0$. Then the extraction of {\it a.e.} convergent subsequence of $\left(u^{\varepsilon}\right)$ by applying Velocity Averaging lemma follows from similar arguments of \cite{Ramesh_Mondal}. Let $\Omega=\mathbb{R},\, T=\infty$. In \cite{Markati2}, Markati considers the viscosity term of the form $\varepsilon\left(B(u^{\varepsilon})\,u_{x}\right)_{x}$ and prove that the {\it a.e.} limit of the viscosity approximations is the unique Kruzhkov's entropy solution to the corresponding scalar conservation laws. In \cite{Markati1}, the Markati and Natalini study viscosity problem with degenarate gradient dependent viscosity of the form $\varepsilon\left(B(u_{x})\right)_{x}$. They also prove the convergence of the viscosity approximations to the unique Kruzhkov's entropy solution of  scalar conservation laws using Compensated Compactness. In \cite{Ramesh}, we have established the BV-estimates for quasilinear viscous approximations with BV initial data. In \cite{Ramesh_Mondal}, we have shown that the {\it a.e.} limit of quasiliear viscous approximations is the unique entropy solution in the sense of Otto \cite{MR1387428} with \textbf{Hypothesis D}. In this article, we prove the convergence of quasilinear viscous approximations $\left(u^{\varepsilon}\right)$ to Otto's  entropy solution \cite{MR1387428} for any space dimension and for any bounded domain $\Omega$ of $\mathbb{R}^{d}$, $0<T<\infty$ and assuming $B\geq 0$ as mentioned in \textbf{Hypothesis G}.\\
In the proof of Theorem \ref{Deg.paper2.compensatedcompactness.theorem1}, we need to first prove that the {\it a.e.} limit of $\left(u^{\varepsilon}\right)$ is a weak solution for the IBVP \eqref{ivp.cl}. Then we secondly prove that the {\it a.e.} limit of $\left(u^{\varepsilon}\right)$ satisfies an entropy inequality mentioned in the definition of Otto's entropy solution. Since we work with $B\geq 0$ in place of $B\geq r>0$, we need to use the boundedness of the sequence $$\left(\sqrt{\varepsilon}\Big\| \sqrt{B\left(u^{\varepsilon,\delta}\right)}\,\frac{\partial u^{\varepsilon,\delta}}{\partial x_{j}}\Big\|_{L^{2}(\Omega_{T})}\right)$$
in place of the boundedness of the sequence  $$\left(\sqrt{\varepsilon}\Big\| \frac{\partial u^{\varepsilon,\delta}}{\partial x_{j}}\Big\|_{L^{2}(\Omega_{T})}\right)$$
 for $j=1,2,\cdots,d$ and we encounter double limit passing situations\,\,$\left(\mbox{firstly}\,\delta\to 0\,\,\mbox{followed by}\,\,\varepsilon\to 0\right)$\,\, in the proof of weak solution and entropy inequality. We handle it with careful clever computaions in the proof. This double limit passing situation is not present in the proof of the entropy solution when we worked with \textbf{Hypothesis D} in \cite{Ramesh_Mondal}. The rest of the proof of Theorem \ref{Deg.paper2.compensatedcompactness.theorem1} follows from \cite{Ramesh_Mondal}. 
\vspace{0.1cm}\\
The plan for the paper are the following. In Section \ref{NonDegenerate.Section.1}, we pose an IBVP for viscosity problem with non-degenerate viscosity $B+\delta$. Then we prove existence, uniqueness and useful properties of solutions of IBVP \eqref{regularized.IBVP} for viscosity problems. We also prove the compactness properties of sequence of solutions to IBVP \eqref{regularized.IBVP} with respect to the index $\delta$. In Section \ref{NonDegenerate.Section.2}, we prove the existence of {\it a.e.} convergent subsequence of the quasilinear viscous approximations $\left(u^{\varepsilon}\right)$ using Velocity Averaging lemma. In Section \ref{Degenerate.Section.3}, we show that the {\it a.e.} limit of $\left(u^{\varepsilon}\right)$ is the unique entropy solution of IBVP \eqref{ivp.cl} in the sense of Otto \cite{MR1387428}.   
\section{Existence, uniqueness, maximum principle and derivative estimates}\label{NonDegenerate.Section.1}
In this entire article, we work with $B$ as mentioned in Hypothesis G from this section onwards. For fixed $\varepsilon>0$, consider the IBVP for the generalized viscosity problem 
\begin{subequations}\label{regularized.IBVP}
	\begin{eqnarray}
	u^{\varepsilon,\delta}_{t} + \nabla \cdot f(u^{\varepsilon,\delta}) = \varepsilon\,\nabla\cdot\left(\left(B(u^{\varepsilon,\delta})+\delta\right)\,\nabla u^{\varepsilon,\delta}\right)
	&\mbox{in }\Omega_{T},\label{regularized.IBVP.a} \\
	u^{\varepsilon,\delta}(x,t)= 0&\,\,\,\,\mbox{on}\,\, \partial \Omega\times(0,T),\label{regularized.IBVP.b}\\
	u^{\varepsilon,\delta}(x,0) = u_{0\delta}^{\varepsilon}(x)& x\in \Omega,\label{regularized.IBVP.c}
	\end{eqnarray}
\end{subequations}
where $\delta>0$. Observe that since $B\geq 0$, the function $B(u^{\varepsilon,\delta}) +\delta$ satisfy the Hypothesis D as $B(u^{\varepsilon,\delta}) +\delta\geq \delta>0$ and the sequence $\left(u_{0\delta}^{\varepsilon}\right)$ is constructed in Lemma \ref{De.hypothesisDinitialdatalem}.\\
\vspace{0.1cm}\\
We begin this section by constructing sequences of approximations $\left(u_{0\varepsilon}\right)$ for the initial data $u_{0}$. The sequence $\left(u_{0\varepsilon}\right)$ mentioned in \textbf{Hypothesis G} are constructed in the following result in view of the discussions from \cite[p.31-p.35]{MR2309679}. We prove the result for any space dimension $d\in\mathbb{N}$. 	
\begin{lemma}\label{hypothesisDinitialdatalem}
	{\rm Let $u_{0}\in L^{\infty}(\Omega)$ and $1\leq p<\infty$. Then there exists a sequence $\left(u_{0\varepsilon}\right)$ in $\mathcal{D}(\Omega)$ and $A> 0$ such that $\|u_{0\varepsilon}\|\leq A$ and $u_{0\varepsilon}\to u_{0}$ in $L^{p}(\Omega)$ as $\varepsilon\to 0$.} 	
\end{lemma}	
Applying Lemma \ref{hypothesisDinitialdatalem} with $u_{0}=u_{0,\varepsilon}$ we conclude the following result.
\begin{lemma}\label{De.hypothesisDinitialdatalem}
	{\rm Let $\delta>0$ and $1\leq p<\infty$. Then, for every $\varepsilon>0$, there exists a sequence $\left(u_{0\varepsilon}^{\delta}\right)$ in $\mathcal{D}(\Omega)$ and $A> 0$ such that $\|u_{0\delta}^{\varepsilon}\|\leq A$ and $u_{0\delta}^{\varepsilon}\to u_{0\varepsilon}$ in $L^{p}(\Omega)$ as $\delta\to 0$.}	
\end{lemma}	
\subsection{Existence and Uniqueness of Solutions}
Applying a result from \cite{lad-etal_68a}, we conclude the following existence of a unique classical solution for IBVP \eqref{regularized.IBVP}.
\begin{theorem}\cite[p.452]{lad-etal_68a}[\textbf{Unique Classical Solution}]\label{ExistenceofClassicalLadyzenskajap452}
	{\rm Let $B$ be as in Hypothesis G. Let $f$, $B+\delta$, $u_{0\delta}^{\varepsilon}$ satisfy Hypothesis D. Then there exists a unique solution $u^{\varepsilon,\delta}$ of generalized viscosity problem \eqref{regularized.IBVP} in the space
		$C^{2+\beta,\frac{2+\beta}{2}}(\overline{\Omega_T})$. Further, for each $i=1,2,\cdots,d$ the second order partial derivatives $u^{\varepsilon,\delta}_{x_i t}$ belong to $L^2(\Omega_T)$.}  
\end{theorem}
\subsection{Estimates on Solutions and its derivatives}
We now recall the maximum principle from \cite[p.12]{Ramesh}, {\it i.e.,}
\begin{theorem}[Maximum principle]\label{chap3thm1}
	{\rm Let $f:\mathbb{R}\to\mathbb{R}^{d}$ be a $C^{1}$ function and $u_{0}\in L^{\infty}(\Omega)$. Then any solution $u^{\varepsilon}$ of generalized viscosity problem \eqref{regularized.IBVP} satisfies the bound
		\begin{equation}\label{eqnchap303}
		||u^{\varepsilon}(\cdot,t)||_{L^{\infty}(\Omega)}\,\leq\,||u_{0}||_{L^{\infty}(\Omega)}\,a.e.\,\,t\in(0,T).
		\end{equation}
	}
\end{theorem}
Let $B$ be as in Hypothesis G. For $\delta>0$, $B+\delta$ satisfy Hypothesis D. Applying Theorem \ref{chap3thm1} to regularized viscosity problem \eqref{regularized.IBVP} and using $\|u_{0\delta}^{\varepsilon}\|_{L^{\infty}(\Omega)}
\leq A$, we conclude the next result.
\begin{theorem}\label{regularized.chap3thm1}
	{\rm Let $B$ be as in Hypothesis G. Let $f,\,B+\delta,\,$ and $u_{0\delta}^{\varepsilon}$ satisfy Hypothesis D. Then any solution $u^{\varepsilon,\delta}$ of generalized viscosity problem \eqref{regularized.IBVP}  satisfies the bound
		\begin{equation}\label{regularized.eqnchap303}
		||u^{\varepsilon,\delta}(\cdot,t)||_{L^{\infty}(\Omega_{T})}\,\leq\,A.
		\end{equation}
	}
\end{theorem}
The viscous approximations $\left(u^{\varepsilon,\delta}\right)$ obtained in
Theorem \ref{ExistenceofClassicalLadyzenskajap452} satisfies the following estimate. This is a very useful result.
\begin{theorem}\label{Compactness.lemma.1}
	{\rm Let $B$ be as in Hypothesis G. Let $f,\,\,B+\delta,\,\,u_{0\delta}^{\varepsilon}$ satisfy Hypothesis D. Let $u^{\varepsilon,\delta}$ be the unique solution to generalized 
		viscosity problem \eqref{regularized.IBVP}. Then 
		\begin{eqnarray}\label{uniformnot.compactness.eqn1a}
		\displaystyle\sum_{j=1}^{d} \,\left(\sqrt{\varepsilon}\Big\| \frac{\partial u^{\varepsilon,\delta}}{\partial x_{j}}\Big\|_{L^{2}(\Omega_{T})}\right)^{2} \leq\frac{1}{2\delta}\|u_{0\delta}^{\varepsilon}\|^{2}_{L^{2}(\Omega)}\leq\frac{1}{2\delta} A^2\,\,\mbox{Vol}(\Omega).
		\end{eqnarray}
	}
\end{theorem}
A proof of Theorem \ref{Compactness.lemma.1} follows from \cite{Ramesh}.\\
Following the proof of Theorem 3.3 of \cite{Ramesh} and applying $B(u^{\varepsilon,\delta})+\delta\geq B(u^{\varepsilon,\delta})$, we conclude the following result.
\begin{theorem}\label{Deg.Compactness.lemma.1}
	{\rm Let $B$ be as in Hypothesis G. Let $f,\,\,B+\delta,\,\,u_{0\delta}^{\varepsilon}$ satisfy Hypothesis D. Let $u^{\varepsilon,\delta}$ be the unique solution to generalized 
		viscosity problem \eqref{regularized.IBVP}. Then 
		\begin{eqnarray}\label{Deg.uniformnot.compactness.eqn1a}
		\displaystyle\sum_{j=1}^{d} \,\left(\sqrt{\varepsilon}\Big\| \sqrt{B\left(u^{\varepsilon,\delta}\right)}\,\frac{\partial u^{\varepsilon,\delta}}{\partial x_{j}}\Big\|_{L^{2}(\Omega_{T})}\right)^{2} \leq\frac{1}{2}\|u_{0\delta}^{\varepsilon}\|^{2}_{L^{2}(\Omega)}\leq\frac{1}{2} A^2\,\,\mbox{Vol}(\Omega).
		\end{eqnarray}
	}
\end{theorem}
\subsection{Compactness of $\left(u^{\varepsilon,\delta}\right)$.}
We recall Theorem 4.5 from \cite{Ramesh_Mondal}. We apply the next result to $\left(u^{\varepsilon,\delta}\right)$ for fixed $\varepsilon>0$. 
\begin{theorem}\label{Kinetic.Compactness.Result}
	{\rm Let $B$ be as in \textbf{Hypothesis G} and $1<p\leq 2$. Let $f,\,B+\delta,\,u_{0\delta}^{\varepsilon}$ satisfy \textbf{Hypothesis D}. Then, the sequence of  viscous approximations $\left(u^{\varepsilon,\delta}\right)$ which are solutions of IBVP \eqref{regularized.IBVP} lies in a compact subset of  $L^{p}_{loc}\left(\Omega_{T}\right)$.}	
\end{theorem} 
For fixed $\varepsilon>0$, aplying Theorem \ref{Kinetic.Compactness.Result} to $\left(u^{\varepsilon,\delta}\right)$, we conclude the following result.
\begin{theorem}\label{Degenerate.Kinetic.Compactness.Result}
{\rm Let $B$ be as in \textbf{Hypothesis G}. Let $f,\,B+\delta,\,u_{0\delta}^{\varepsilon}$ satisfy \textbf{Hypothesis D}. There exists a subsequence of $\left(u^{\varepsilon,\delta}\right)$ and a function $u^{\varepsilon}$ in $L^{p}_{loc}\left(\Omega_{T}\right)$ such that $u^{\varepsilon,\delta}\to u^{\varepsilon}$ as $\delta\to 0$ in $L^{p}_{loc}\left(\Omega_{T}\right)$. }	
\end{theorem} 
\begin{remark}
	Since $\mbox{Vol}\left(\Omega_{T}\right)<\infty$, therefore we have $u^{\varepsilon,\delta}\to u^{\varepsilon}$ as $\delta\to 0$ in $L^{p}_{loc}\left(\Omega_{T}\right)$ for $1\leq p\leq 2$.
\end{remark}
\section{Compactness of quasilinear viscous approximations $\left(u^{\varepsilon}\right)$.}\label{NonDegenerate.Section.2}
The aim of this section is to show the existence of {\it a.e.} convergent subsequence of the quasilinear viscous approximations $\left(u^{\varepsilon}\right)$. We now introduce the sign function and the Kruzhkov entropy-entropy flux pairs. Denote 
$$\mbox{sg}(s):=\begin{cases}
1\,\,\,\mbox{if}\,\,\,s>0\\
0\,\,\,\mbox{if}\,\,\,s=0\\
-1\,\,\,\mbox{if}\,\,\,s<0
\end{cases}$$

For $c\in\mathbb{R}$, the family of Kruzhkov entropy and entropy fluxes are $\eta(u;c):=\left|u-c\right|$ and $j\in\left\{1,2,\cdots,d\right\}$, $q_{j}(u;c)=\mbox{sg}\left(u-c\right)\left(f_{j}(u)-f_{j}(c)\right)$.\\ 
\vspace{0.1cm}\\
We prove the following estimate of the sequence of solutions $\left(u^{\varepsilon}\right)$ to \eqref{De.regularized.IBVP}. 
\begin{lemma}\label{Degenerate.Maximum.Principle}
{\rm Let $f,\,B,\,u_{0\varepsilon}$ satisfy the \textbf{Hypothesis G}. Then $\left(u^{\varepsilon}\right)$ satisfy 
$\|u^{\varepsilon}\|_{L^{\infty}\left(\Omega_{T}\right)}\leq A$.}
\end{lemma}
\begin{proof}
Applying Theorem \ref{Degenerate.Kinetic.Compactness.Result}, we conclude that for {\it a.e.} $(x,t)\in\Omega_{T}$, we have $u^{\varepsilon,\delta}\to u^{\varepsilon}$ as $\delta\to 0$ (A subsequence). Again, applying Theorem \ref{regularized.chap3thm1}, we have
$$-A\leq u^{\varepsilon,\delta}(x,t)\leq A. $$
Therefore, passing to the limit as $\delta\to 0$, we obtain that $u^{\varepsilon}$ satisfy
$$-A\leq u^{\varepsilon}(x,t)\leq A$$	
for {\it a.e.} $(x,t)\in\Omega_{T}$. This completes the proof $\blacksquare$
\end{proof}	\\
Applying Theorem \ref{Compactness.lemma.1}, Lemma \ref{Degenerate.Maximum.Principle}, Dominated convergence theorem and IBVP \eqref{regularized.IBVP}, we conclude the following result.
\begin{proposition}\label{Deg.Definition.of.solution}
{\rm Let $f,\,B,\,u_{0\varepsilon}$ satisfy the \textbf{Hypothesis G}. Then, for all $v\in H^{1}_{0}\left(\Omega_{T}\right)$, $\left(u^{\varepsilon}\right)$ satisfy 
\begin{equation}\label{Definition.Solution.Deg1}
\int_{\Omega_{T}}\,u^{\varepsilon}\,\frac{\partial v}{\partial t}\,dx\,dt +\displaystyle\sum_{j=1}^{d}\int_{\Omega_{T}}\,f_{j}\left(u^{\varepsilon}\right)\,\frac{\partial v}{\partial x_{j}}\,dx\,dt =\varepsilon\displaystyle\lim_{\delta\to 0}\left(\displaystyle\sum_{j=1}^{d}\int_{\Omega_{T}}\,B(u^{\varepsilon,\delta})\frac{\partial u^{\varepsilon,\delta}}{\partial x_{j}}\,\frac{\partial v}{\partial x_{j}}\,dx\,dt\right).
\end{equation}
}	
\end{proposition}
\begin{remark}\label{Degenerate.ConceptofSolution}
{\rm The equation \eqref{Definition.Solution.Deg1} along with $u^{\varepsilon}\left(x,0\right)=u_{0\varepsilon}(x)$ is the concept of solution by which $u^{\varepsilon}$ satisfies the equation \eqref{De.regularized.IBVP.a}. A measure representation of this concept of solution can also be written.}\\	
\end{remark}
\vspace{0.1cm}
Denote
$$\chi_{u^{\varepsilon}}(c):= \begin{cases}
1 ,\,\,\mbox{if}\,\,u^{\varepsilon}<c< 0,\\
-1 ,\,\,\mbox{if}\,\,0< c< u^{\varepsilon},\nonumber\\
0 ,\,\,\mbox{otherwise}\nonumber
\nonumber
\end{cases}\\$$
\begin{theorem}\label{Degenerate.measure.1}
Let $f,\,u_{0\varepsilon},\,B$ satisfy \textbf{Hypothesis G}. Let $\left(u^{\varepsilon}\right)$ be the sequence of solutions of \eqref{De.regularized.IBVP}.
\begin{enumerate}
	\item Then for all $\phi\in\mathcal{D}\left(\Omega_{T}\right)$, $u^{\varepsilon}$ satisfies
	\begin{equation}\label{Degenerate.equation1.compactness}
	\left<\displaystyle\lim_{\delta\to 0}\left(\varepsilon\,\eta^{\prime}\left(u^{\varepsilon,\delta};c\right)\displaystyle\sum_{j=1}^{d}\frac{\partial}{\partial x_{j}}\left(B(u^{\varepsilon,\delta})\,\frac{\partial u^{\varepsilon,\delta}}{\partial x_{j}}\right)\right),\phi\right> =\left<\frac{\partial}{\partial t}\eta(u^{\varepsilon};c) + \displaystyle\sum_{j=1}^{d}\,\frac{\partial}{\partial x_{j}}q_{j}(u^{\varepsilon};c)\,,\phi\right>
	\end{equation}
	\item \begin{equation}\label{degenerate.KineticFormulation.Equation612}
	\frac{\partial \chi^{\varepsilon}}{\partial t} +\displaystyle\sum_{j=1}^{d}f_{j}^{\prime}(c)\,\frac{\partial\chi^{\varepsilon}}{\partial x_{j}} = \frac{\partial}{\partial c}\left(\displaystyle\lim_{\delta\to 0}\left(\frac{\varepsilon}{2}\,\displaystyle\sum_{j=1}^{d}\eta^{\prime}(u^{\varepsilon,\delta};c)\frac{\partial}{\partial x_{j}}\left(B(u^{\varepsilon,\delta})\frac{\partial u^{\varepsilon,\delta}}{\partial x_{j}}\right)\right)\right)
	\end{equation}
	in  $\mathcal{D}^{\prime}\left(\Omega\times\mathbb{R}\times(0,T)\right)$	
\end{enumerate}
 	
\end{theorem}	
\begin{proof}
A proof of Theorem \ref{Degenerate.measure.1} follows from the computation of Proposition 4.1 of \cite{Ramesh_Mondal} and passing to the limit as $\delta\to 0$. For the completeness of the proof, we repeat the computation of Proposition 4.1 from \cite{Ramesh_Mondal}. \\	
Let $G:\mathbb{R}\to\mathbb{R}$ be a $C^{\infty}-$ function such that 
$$G(x):=\left|x\right|\,\,\mbox{for}\,\,\left|x\right|\geq 1,\,\,G^{\prime\prime}\geq 0.$$
For $n\in\mathbb{N}$, denote 
\begin{equation*}
G_{n}(x):=\frac{1}{n}\,G\left(n\,\left(x-c\right)\right).$$
Then $G_{n}(x)\to \left|x-c\right|$ as $n\to\infty$. We now compute 
$$G^{\prime}_{n}(x):=\begin{cases}
\frac{d}{dx}\left(\left|x-c\right|\right),\,\,\mbox{if}\,\,\left|x-c\right|\geq\frac{1}{n},\\
G^{\prime}\left(n\left(x-c\right)\right),\,\mbox{if}\,\left|x-c\right|<\frac{1}{n}.
\end{cases}
\end{equation*}
Then we have
\begin{equation*}
G^{\prime}_{n}(x)=
\begin{cases}
1\,\,\mbox{if}\,\,x\geq c\,\,\left|x-c\right|\geq\frac{1}{n},\\
-1\,\,\mbox{if}\,\,x < c\,\,\left|x-c\right|\geq\frac{1}{n},\\
G^{\prime}\left(n\left(x-c\right)\right),\,\mbox{if}\,\left|x-c\right|<\frac{1}{n}.
\end{cases}
\end{equation*}
Therefore 
$$G_{n}^{\prime}(x)\to \mbox{sg}(x-c)\,\,\mbox{as}\,\,n\to\infty.$$
Multiplying \eqref{regularized.IBVP.a} by $G^{\prime}_{n}(u^{\varepsilon,\delta})$ to get 
\begin{equation*}
G^{\prime}_{n}\left(u^{\varepsilon,\delta}\right)\,\frac{\partial u^{\varepsilon,\delta}}{\partial t} + \displaystyle\sum_{j=1}^{d}\,G^{\prime}_{n}\left(u^{\varepsilon,\delta}\right)\,\frac{\partial}{\partial x_{j}}f_{j}\left(u^{\varepsilon,\delta}\right)=\varepsilon\,G^{\prime}_{n}\left(u^{\varepsilon,\delta}\right)\displaystyle\sum_{j=1}^{d}\frac{\partial}{\partial x_{j}}\left(B(u^{\varepsilon,\delta})\,\frac{\partial u^{\varepsilon,\delta}}{\partial x_{j}}\right)
\end{equation*}	
Applying chain rule, we get
\begin{equation}\label{KineticFormulation.Equation1}
\frac{\partial}{\partial t}\left( G_{n}\left(u^{\varepsilon,\delta}\right)\right) + \displaystyle\sum_{j=1}^{d}\,G^{\prime}_{n}\left(u^{\varepsilon,\delta}\right)\,f^{\prime}_{j}\left(u^{\varepsilon,\delta}\right)\frac{\partial u^{\varepsilon,\delta}}{\partial x_{j}}=\varepsilon\,G^{\prime}_{n}\left(u^{\varepsilon,\delta}\right)\displaystyle\sum_{j=1}^{d}\frac{\partial}{\partial x_{j}}\left(B(u^{\varepsilon})\,\frac{\partial u^{\varepsilon,\delta}}{\partial x_{j}}\right)
\end{equation}
Denote
$$q_{n}(z):=\int_{k}^{z}\,G^{\prime}_{n}\left(v\right)\,f^{\prime}(v)\,dv.$$
Using $q_{n}$ in \eqref{KineticFormulation.Equation1}, we obtain
\begin{equation*}
\frac{\partial}{\partial t}\left( G_{n}\left(u^{\varepsilon,\delta}\right)\right) + \displaystyle\sum_{j=1}^{d}\,q_{nj}^{\prime}(u^{\varepsilon,\delta})\frac{\partial u^{\varepsilon,\delta}}{\partial x_{j}} = \varepsilon\,G^{\prime}_{n}\left(u^{\varepsilon,\delta}\right)\displaystyle\sum_{j=1}^{d}\frac{\partial}{\partial x_{j}}\left(B(u^{\varepsilon,\delta})\,\frac{\partial u^{\varepsilon,\delta}}{\partial x_{j}}\right) 
\end{equation*}
Passing to the limit as $n\to\infty$ in the sense of distribution, we obtain 
\begin{equation}\label{KineticFormulation.Equation2}
\frac{\partial}{\partial t}\eta(u^{\varepsilon,\delta};c) + \displaystyle\sum_{j=1}^{d}\,\frac{\partial}{\partial x_{j}}q_{j}(u^{\varepsilon,\delta};c) = \varepsilon\,\eta^{\prime}\left(u^{\varepsilon,\delta};c\right)\displaystyle\sum_{j=1}^{d}\frac{\partial}{\partial x_{j}}\left(B(u^{\varepsilon,\delta})\,\frac{\partial u^{\varepsilon,\delta}}{\partial x_{j}}\right)\,\,\mbox{in}\,\,\mathcal{D}^{\prime}\left(\Omega\times (0,T)\right). 
\end{equation}
Applying of Theorem \ref{regularized.chap3thm1} and Dominated convergence theorem, we pass to the limit as $\delta\to 0$ and we obtain 
\begin{equation}\label{Degenerate.KineticFormulation.Equation2}
\frac{\partial}{\partial t}\eta(u^{\varepsilon};c) + \displaystyle\sum_{j=1}^{d}\,\frac{\partial}{\partial x_{j}}q_{j}(u^{\varepsilon};c) = \displaystyle\lim_{\delta\to 0}\left(\varepsilon\,\eta^{\prime}\left(u^{\varepsilon,\delta};c\right)\displaystyle\sum_{j=1}^{d}\frac{\partial}{\partial x_{j}}\left(B(u^{\varepsilon,\delta})\,\frac{\partial u^{\varepsilon,\delta}}{\partial x_{j}}\right)\right)\,\,\mbox{in}\,\,\mathcal{D}^{\prime}\left(\Omega\times (0,T)\right). 
\end{equation}
From \eqref{KineticFormulation.Equation2}, we obtain
\begin{equation}\label{KineticFormulation.Equation3}
\begin{split}
\frac{\partial}{\partial t}\left[\eta(u^{\varepsilon};c)-\eta(0;c)\right] &+ \displaystyle\sum_{j=1}^{d}\,\frac{\partial}{\partial x_{j}}\left[q_{j}(u^{\varepsilon};c)-q_{j}(0;c)\right]\\ 
&=\displaystyle\lim_{\delta\to 0}\left(\varepsilon\,\eta^{\prime}\left(u^{\varepsilon,\delta};c\right)\displaystyle\sum_{j=1}^{d}\frac{\partial}{\partial x_{j}}\left(B(u^{\varepsilon,\delta})\,\frac{\partial u^{\varepsilon}}{\partial x_{j}}\right)\right)\,\,\mbox{in}\,\,\mathcal{D}^{\prime}\left(\Omega\times (0,T)\right). 
\end{split}
\end{equation}
Differentiating \eqref{KineticFormulation.Equation3} with respect to $c$ we get
\begin{equation}\label{KineticFormulation.Equation3ABC1}
\begin{split}
\frac{\partial}{\partial t}\left(\frac{\partial}{\partial c}\left(\eta(u^{\varepsilon};c)-\eta(0;c)\right)\right) &+ \displaystyle\sum_{j=1}^{d}\,\frac{\partial}{\partial x_{j}}\left(\frac{\partial}{\partial c}\left(q_{j}(u^{\varepsilon};c)-q_{j}(0;c)\right)\right)\\ 
&=\frac{\partial}{\partial c}\left(\displaystyle\lim_{\delta\to 0}\left(\varepsilon\,\eta^{\prime}\left(u^{\varepsilon,\delta};c\right)\displaystyle\sum_{j=1}^{d}\frac{\partial}{\partial x_{j}}\left(B(u^{\varepsilon,\delta})\,\frac{\partial u^{\varepsilon,\delta}}{\partial x_{j}}\right)\right)\right)\,\,\mbox{in}\,\,\mathcal{D}^{\prime}\left(\Omega\times (0,T)\right). 
\end{split}
\end{equation}
Note that $\eta(u^{\varepsilon};c)=\left|u^{\varepsilon}-c\right|$, $\eta(0;c)=\left|c\right|$. Denote
$$P\left(u^{\varepsilon};c\right):=\left|u^{\varepsilon}-c\right|-|c|.$$
We now compute the derivative of $P$ with respect to $c$.
\begin{eqnarray}\label{KineticFormulation.Equation4}
\frac{\partial}{\partial c}P(u^{\varepsilon};c) &=&\frac{\partial}{\partial c}\left(|u^{\varepsilon}-c|-|c|\right)\nonumber\\
&=& \frac{\partial}{\partial c}|u^{\varepsilon}-c|-\frac{\partial}{\partial c}|c|\nonumber\\
&=& \mbox{sg}(u^{\varepsilon}-c)-\mbox{sg }(c)\nonumber\\
\end{eqnarray}
For $j=1,2,\cdots,d$, denote
\begin{eqnarray}\nonumber
Q_{j}(u^{\varepsilon};c) &:=& q_{j}(u^{\varepsilon};c)-q_{j}(0;c),\nonumber\\
&=& \mbox{sg}\,(u^{\varepsilon}-c)\left(f_{j}(u^{\varepsilon})-f_{j}(c)\right) + \mbox{sg}(c)\,\left(f_{j}(0)-f_{j}(c)\right).\nonumber	
\end{eqnarray}	
Let us compute the derivative of $Q_{j}^{\varepsilon}$ with respect to $c$.
\begin{equation*}
\begin{split}
\frac{\partial}{\partial c}Q_{j}(u^{\varepsilon};c)
= 2\,\left(f_{j}(u^{\varepsilon}(x,t))-f_{j}(u^{\varepsilon}(x,t))\right)\,\delta_{c=u^{\varepsilon}(x,t)} -f_{j}^{\prime}(c)\,\mbox{sg}\,(u^{\varepsilon}-c)\\ +2\,\left(f_{j}(0)-f_{j}(0)\right)\,\delta_{c=0} -f_{j}^{\prime}(c)\,\mbox{sg}\left(c\right).	
\end{split} 
\end{equation*}
Therefore we have
\begin{equation}\label{KineticFormulation.Equation5}
\frac{\partial}{\partial c}Q_{j}(u^{\varepsilon};c)
=-f_{j}^{\prime}(c)\left(\,\,\mbox{sg}\left(u^{\varepsilon}-c\right)+\,\mbox{sg}(c)\,\right).
\end{equation}
Let $\phi\in\mathcal{D}\left(\Omega\times\mathbb{R}\times (0,T)\right)$. For $j\in\left\{1,2,\cdots,d\right\}$, we compute
\begin{eqnarray}\label{KineticFormulation.Equation3ABC2}
\left<\frac{\partial^{2}}{\partial x_{j}\partial c}Q_{j}(u^{\varepsilon};c),\phi\right> &=& -\left<\frac{\partial}{\partial c}Q_{j}(u^{\varepsilon};c),\frac{\partial\phi}{\partial x_{j}}\right>\nonumber\\
&=&  \Big<f_{j}^{\prime}(c)\left(\mbox{sg}\left(u^{\varepsilon}-c\right)+\mbox{sg}(c)\right), \frac{\partial\phi}{\partial x_{j}}\Big>\nonumber\\ 
&=& \Big<f_{j}^{\prime}(c)\frac{\partial}{\partial x_{j}}\left(\mbox{sg}\left(u^{\varepsilon}-c\right)+\mbox{sg}(c)\right),\phi\Big>\nonumber\\
&=&  \Big<f_{j}^{\prime}(c)\frac{\partial}{\partial x_{j}}\left(\mbox{sg}\left(u^{\varepsilon}-c\right)-\mbox{sg}(c)\right),\phi\Big>\nonumber\\
&=& \Big<f_{j}^{\prime}(c)\left(\mbox{sg}\left(u^{\varepsilon}-c\right)-\mbox{sg}(c)\right),\frac{\partial\phi}{\partial x_{j}}\Big>.
\end{eqnarray}
Denote
$$\chi^{\varepsilon}(x,t ;c):=\chi_{u^{\varepsilon}}(c),$$
where
$$\chi_{u^{\varepsilon}}(c):= \begin{cases}
1 ,\,\,\mbox{if}\,\,u^{\varepsilon}<c< 0,\\
-1 ,\,\,\mbox{if}\,\,0< c< u^{\varepsilon},\nonumber\\
0 ,\,\,\mbox{otherwise}\nonumber
\nonumber
\end{cases}\\$$
Also observe that $\int_{-\infty}^{\infty}\,\chi_{u^{\varepsilon}}(c)\,dc=-u^{\varepsilon}$ and $\chi^{\varepsilon}=\chi_{u^{\varepsilon}}(c)$. Dividing \eqref{KineticFormulation.Equation3ABC1} by 2 and applying the above computations, we get \eqref{degenerate.KineticFormulation.Equation612}$\blacksquare$ \\	
\end{proof}	
\begin{remark}
{\rm Denote 
$$p_{\mbox{min}}^{\varepsilon,\delta}:=\mbox{min}\left\{\displaystyle\min_{(x,t)\in\overline{\Omega_{T}}}\frac{\varepsilon}{2}\,\displaystyle\sum_{j=1}^{d}\,1\cdot\frac{\partial}{\partial x_{j}}\left(B(u^{\varepsilon,\delta})\frac{\partial u^{\varepsilon,\delta}}{\partial x_{j}}\right),\,\displaystyle\min_{(x,t)\in\overline{\Omega_{T}}}\frac{\varepsilon}{2}\,\displaystyle\sum_{j=1}^{d}\,\left(-1\right)\cdot\frac{\partial}{\partial x_{j}}\left(B(u^{\varepsilon,\delta})\frac{\partial u^{\varepsilon,\delta}}{\partial x_{j}}\right),\,0\right\}$$\\

Since $\mbox{range}\left(\eta^{\prime}\left(u^{\varepsilon,\delta};c\right)\right)=\mbox{sgn}\left(u^{\varepsilon,\delta}-c\right)=\left\{-1,0,1\right\}$, $p_{\mbox{min}}^{\varepsilon,\delta}$ is independent of $c$ variable and we have 
$$\frac{\varepsilon}{2}\,\displaystyle\sum_{j=1}^{d}\eta^{\prime}(u^{\varepsilon,\delta};c)\frac{\partial}{\partial x_{j}}\left(B(u^{\varepsilon,\delta})\frac{\partial u^{\varepsilon,\delta}}{\partial x_{j}}\right)\geq p_{\mbox{min}}^{\varepsilon,\delta}.$$
Applying $\limsup$ on both sides of the inequality as $\delta\to 0$, we obtain
$$\displaystyle\limsup_{\delta\to 0}\left(\frac{\varepsilon}{2}\,\displaystyle\sum_{j=1}^{d}\eta^{\prime}(u^{\varepsilon,\delta};c)\frac{\partial}{\partial x_{j}}\left(B(u^{\varepsilon,\delta})\frac{\partial u^{\varepsilon,\delta}}{\partial x_{j}}\right)\right)\geq p_{\mbox{min}}^{\varepsilon,\delta}.$$

 and applying \eqref{Degenerate.KineticFormulation.Equation2}, we get
\begin{equation*}
\begin{split}
\frac{\partial}{\partial c}\left(\displaystyle\lim_{\delta\to 0}\left(\frac{\varepsilon}{2}\,\displaystyle\sum_{j=1}^{d}\eta^{\prime}(u^{\varepsilon,\delta};c)\frac{\partial}{\partial x_{j}}\left(B(u^{\varepsilon,\delta})\frac{\partial u^{\varepsilon}}{\partial x_{j}}\right)\right)\right)\\
=\frac{\partial}{\partial c}\left(\displaystyle\limsup_{\delta\to 0}\left(\frac{\varepsilon}{2}\,\displaystyle\sum_{j=1}^{d}\eta^{\prime}(u^{\varepsilon,\delta};c)\frac{\partial}{\partial x_{j}}\left(B(u^{\varepsilon,\delta})\frac{\partial u^{\varepsilon}}{\partial x_{j}}\right)\right)\right)\\
=\frac{\partial}{\partial c}\left(\displaystyle\limsup_{\delta\to 0}\left(\frac{\varepsilon}{2}\,\displaystyle\sum_{j=1}^{d}\eta^{\prime}(u^{\varepsilon,\delta};c)\frac{\partial}{\partial x_{j}}\left(B(u^{\varepsilon,\delta})\frac{\partial u^{\varepsilon,\delta}}{\partial x_{j}}\right)\right)-p^{\varepsilon,\delta}_{\mbox{min}}\right)
\end{split}
\end{equation*}
Therefore \eqref{KineticFormulation.Equation3ABC1} defines a kinetic formulation with $\chi^{\varepsilon}(x,t ;c):=\chi_{u^{\varepsilon}}(c)$ according to the definition of \cite[p.170]{Lions}.
}
\end{remark}
Denote
$$\m_{1}^{\varepsilon}:=\displaystyle\lim_{\delta\to 0}\left(\frac{\varepsilon}{2}\,\displaystyle\sum_{j=1}^{d}\eta^{\prime}(u^{\varepsilon,\delta};c)\frac{\partial}{\partial x_{j}}\left(B(u^{\varepsilon,\delta})\frac{\partial u^{\varepsilon,\delta}}{\partial x_{j}}\right)\right).$$
For a real number $M>0$, denote
$$\Omega_{MT}:=\Omega\times (-M,M)\times (0,T).$$
$$C^{1}_{0}\left(\overline{\Omega_{MT}}\right):=\left\{v\in C^{1}\left(\overline{\Omega_{MT}}\right)\,:\,\,v(x,c,t)=0\,\,\,\mbox{on}\,\partial\Omega_{MT}\right\}.$$
For every $M>0$, we firstly prove that the sequence of measures $\left(m^{\varepsilon}_{1}\right)$ is bounded in $\mathcal{M}\left(\Omega\times\left(-M,M\right)\times (0,T)\right)$. This is done in Proposition \ref{Deg.MeasureBounded.1}. Since we want to apply Velocity Averaging lemma on $\mathbb{R}^{d}\times\mathbb{R}\times (0,\infty)$, we secondly extends the sequence of measures $\left(m^{\varepsilon}_{1}\right)$ to  $\mathbb{R}^{d}\times\mathbb{R}\times (0,\infty)$ using standard arguement of Riesz representation theorem of dual spaces preserving the total variation. This extension is denoted by $\widetilde{\widetilde{m^{\varepsilon}_{1}}}$. These results are proved in this Proposition \ref{Deg.Measure.Correction.proposition1} and Proposition \ref{Deg.Measure.Correction.Proposition3}. Then we can assume a standard result from \cite{Lions} that the derivative of $\widetilde{\widetilde{m^{\varepsilon}_{1}}}$ with respect to $c$ can be expressed as the composition of inverses of Riesz potentials which is one of the requirements to apply Velocity Averaging lemma. \\
\begin{proposition}\label{Deg.MeasureBounded.1}
	The sequence of measures $\left(m^{\varepsilon}_{1}\right)$ is bounded in $\mathcal{M}(\overline{\Omega_{MT}})$.
\end{proposition}
\begin{proof}
We prove Proposition \ref{Deg.MeasureBounded.1}	in two steps. In Step 1, we prove that the sequence $\left(m^{\varepsilon}\right)$ is bounded in $\left(C^{1}_{0}\left(\overline{\Omega_{MT}}\right)\right)^{\ast}$. In Step 2, we prove the measure representation of $\left(m^{\varepsilon}\right)$.\\	
\textbf{Step 1:} For $\phi\in C^{1}_{0}\left(\overline{\Omega_{MT}}\right)$, we have 
$$\left<m^{\varepsilon}_{1},\phi\right>=-\int_{\Omega_{MT}}\,|u^{\varepsilon}-c|\,\,\frac{\partial\phi}{\partial t}\,dx\,dc\,dt\,-\,\displaystyle\sum_{j=1}^{d}\int_{\Omega_{MT}}\,\mbox{sg}\left(u^{\varepsilon}-c\right)\left(f_{j}(u^{\varepsilon})-f_{j}(c)\right)\,\,\frac{\partial\phi}{\partial x_{j}}\,dx\,dc\,dt.$$
Therefore it clearly shows that $$\sup\left\{|\left<m^{\varepsilon}_{1},\phi\right>|\,:\,\phi\in C^{1}_{0}\left(\overline{\Omega_{MT}}\right)\right\}$$
is bounded sequence. We equip $C^{1}_{0}\left(\overline{\Omega_{MT}}\right)$ with $\|.\|_{C^{1}}$ with respect to which it is a Banach space. Hence $m^{\varepsilon}_{1}\in\left(C^{1}_{0}\left(\overline{\Omega_{MT}}\right)\right)^{\ast}$.\\
\textbf{Step 2:} Following the proof of Proposition 4.2 of \cite{Ramesh_Mondal}, we get the existence of  measures $\mu^{\varepsilon}_{1}\in\left(\mathcal{M}\left(\overline{\Omega_{MT}}\right)\right)^{d+1}$ such that for $u\in C^{1}_{0}\left(\overline{\Omega_{MT}}\right)$, we get
\begin{eqnarray}\label{measure.correction.equation16}
m^{\varepsilon}_{1}(u) &=& \left<\mu_{10}^{\varepsilon},u\right> +\displaystyle\sum_{j=1}^{d}\left<\mu_{1j}^{\varepsilon},\frac{\partial u^{\varepsilon}}{\partial x_{j}}\right>
\end{eqnarray} 
\hfill{$\blacksquare$}
\end{proof}	
\vspace{0.1cm}\\
Following the proof of Proposition 4.3 of \cite{Ramesh_Mondal}, we conclude
\begin{proposition}\label{Deg.Measure.Correction.proposition1}
	Let $j\in\left\{0,1,2,\cdots,d\right\}$. There exist extension $\widetilde{\mu_{1j}^{\varepsilon}}$ in $\mathcal{M}\left(\mathbb{R}^{d}\times\mathbb{R}\times(0,\infty)\right)$ such that $\|\mu_{1j}^{\varepsilon}\|=\|\widetilde{\mu_{1j}^{\varepsilon}}\|$. 	
\end{proposition}
Applying Proposition \ref{Deg.MeasureBounded.1} and Proposition \ref{Deg.Measure.Correction.proposition1}, we conclude the following result.
\begin{proposition}\label{Deg.Measure.Correction.Proposition3}
	Let $\widetilde{\mu^{\varepsilon}_{1}}:=\left(\widetilde{\mu_{10}^{\varepsilon}},\widetilde{\mu_{11}^{\varepsilon}},\widetilde{\mu_{12}^{\varepsilon}},\cdots,\widetilde{\mu_{1d}^{\varepsilon}}\right)$. Then the sequence of linear functional  $\widetilde{\widetilde{m^{\varepsilon}}}:\left(C\left(\mathbb{R}^{d}\times\mathbb{R}\times(0,\infty)\right)\right)^{d+1}\to\mathbb{R}$ defined by 
	$$\widetilde{\widetilde{\left(m^{\varepsilon}_{1}\right)}}(u)=\displaystyle\sum_{j=0}^{d}\left<\widetilde{\mu_{1j}^{\varepsilon}},\,u_{j}\right>$$ 
	is bounded in $\mathcal{M}\left(\mathbb{R}^{d}\times\mathbb{R}\times(0,\infty)\right)$.	
\end{proposition}
\subsection{Application of Velocity Averaging Lemma}
This Subsection is similar to Subsection 4.3 of \cite{Ramesh_Mondal}. We reproduce a few results and notations from Subsection 4.3 of \cite{Ramesh_Mondal}. We recall the following result which is used to extract a convergent subsequence of $\left(u^{\varepsilon}\right)$.
\begin{theorem}\cite[p.178]{Lions}\label{VelocityAveragingLemma.12}
	{\rm Let $1<p\leq 2$, let $h$ be bounded in $L^{p}\left(\mathbb{R}^{d}\times\mathbb{R}_{c}\times\left(0,\infty\right)\right)$, let $g$ belong to a compact set of $L^{p}\left(\mathbb{R}^{d}\times\mathbb{R}_{c}\times\left(0,\infty\right)\right)$ and $r\geq 0$. We assume that $h$ satisfies 
	\begin{equation}\label{Velocity.AveragingEquation1.A1}
	\frac{\partial h}{\partial t} + \displaystyle\sum_{i=1}^{d}\,f^{\prime}_{i}(c)\frac{\partial h}{\partial x_{i}} =\left(-\Delta_{x,t} +1\right)^{\frac{1}{2}}\left(-\Delta_{c} +1\right)^{\frac{r}{2}}g\,\,\mbox{in}\,\,\mathcal{D}^{\prime}\left(\mathbb{R}^{d}_{x}\times\mathbb{R}_{c}\times (0,\infty)\right),
	\end{equation}	
	where for $i=1,2,\cdots,d$, each $f_{i}^{\prime}\in C^{l,\alpha}_{loc}$ with $l=r$, $\alpha=1$ if $r$ is an integer, $l=[r]$, $\alpha=r-l$, if $r$ is not an integer. Let $\psi\in L^{p^{\prime}}\left(\mathbb{R}_{c}\right)$ having essential compact support.
	Let the following 
	\begin{equation*}
	\begin{split}
	\mbox{meas}\Big\{c\in\mbox{supp}\,\psi,\,\,\tau+\,\left(f_{1}^{\prime}(c),f_{2}^{\prime}(c),\cdots,f_{d}^{\prime}(c)\right)\cdot\xi=0\Big\}=0\,\, \\
	\mbox{for all}\,\left(\tau,\xi\right)\in\mathbb{R}\times\mathbb{R}^{d}\,\, \mbox{with}\,\,\tau^{2} +\left|\xi\right|^{2}=1.
	\end{split}
	\end{equation*}
	holds. Then $\int_{R}\,h\psi\,dc$ belongs to a compact set of $L^{p}_{loc}\left(\mathbb{R}^{d}\times (0,\infty)\right)$.}	
\end{theorem}	
We want to apply Theorem \ref{VelocityAveragingLemma.12} with $h=\chi^{\varepsilon}(c)=\chi_{u^{\varepsilon}(x,t)}(c)$. Observe that $h$ is defined on $\mathbb{R}^{d}\times\mathbb{R}\times (0,\infty)$ but in our case $\chi^{\varepsilon}$ is defined on $\Omega\times\mathbb{R}\times\left(0,T\right)$. In order to get values of $\chi^{\varepsilon}$ on $\mathbb{R}^{d}\times\mathbb{R}\times (0,\infty)$,  we need to extend $u^{\varepsilon}$ outside $\Omega_{T}$. Therefore we define the following extension of $u^{\varepsilon}$. Denote
\begin{equation}\nonumber\\
\widetilde{u^{\varepsilon}}(x,t):=\begin{cases}
u^{\varepsilon}(x,t)\,\,\mbox{if}\,\,(x,t)\in \left(\Omega_{T}\cup\left(\Omega\times\left\{0\right\}\right)\cup\left(\partial\Omega\times(0,T)\right)\right)\nonumber\\
0\,\,\,\,\mbox{if}\,\,\,(x,t)\in\left(\mathbb{R}^{d}\times (0,\infty)\right)\setminus \left(\Omega_{T}\cup\left(\Omega\times\left\{0\right\}\right)\cup\left(\partial\Omega\times(0,T)\right)\right).\nonumber\\
\end{cases}
\end{equation}
\vspace{0.2cm}
Denote $\widetilde{\chi^{\varepsilon}}(x,c,t):=\chi_{\widetilde{u^{\varepsilon}(x,t)}}(c)$.
Observe that $\widetilde{\chi^{\varepsilon}}(x,c;t)=\chi_{\widetilde{u^{\varepsilon}}(x,t)}(c)\equiv 0$ if $(x,c,t)\notin\Omega\times\left[-\|u_{0}\|_{L^{\infty}\left(\Omega\right)},\|u_{0}\|_{L^{\infty}\left(\Omega\right)}\right]\times (0,T)$ as $\|u^{\varepsilon}\|_{L^{\infty}\left(\Omega_{T}\right)}\leq \|u_0\|_{L^{\infty}\left(\Omega\right)}$.\\
\vspace{0.1cm}\\
We now assume the following result from \cite[p.178]{Lions} as the definitions of the inverse of the Bessel potential given in \cite{Samko} and \cite{Stein} are same.
\begin{lemma}\label{Lion.Parthame.Tadmor.p.178}
	{\em For $1<p\leq 2$ and $r> 1+ \frac{d+2}{p}$, there exists a sequence $\left(g^{\varepsilon}\right)$ in a compact set of $L^{p}(\mathbb{R}^{d}\times\mathbb{R}\times(0,\infty))$ such that\, $\frac{\partial \widetilde{\widetilde{{m}^{\varepsilon}}}}{\partial c}=\left(-\Delta_{x,t} + I\right)^{\frac{1}{2}}\left(-\Delta_{c}+I\right)^{\frac{r}{2}}g^{\varepsilon}$ in $\mathcal{D}^{\prime}\left(\mathbb{R}^{d}\times\mathbb{R}\times (0,\infty)\right)$.}
\end{lemma}
Since we are interested in solving IBVP \eqref{regularized.IBVP} in the domain $\Omega_{T}$, we denote
\begin{equation}\nonumber
\widetilde{g^{\varepsilon}}(x,c,t):=\begin{cases}
g^{\varepsilon}(x,c,t)\,\,\mbox{if}\,\,(x,c,t)\in\Omega\times\left(-\|u_{0}\|_{L^{\infty}\left(\Omega\right)},\|u_{0}\|_{L^{\infty}\left(\Omega\right)}\right)\times(0,T),\nonumber\\
0\,\,\,\,\mbox{Otherwise}.\nonumber
\end{cases}
\end{equation}
as $\Omega\times\left\{-\|u_0\|_{L^{\infty}\left(\Omega\right)}\right\}\times (0,T)$, $\Omega\times\left\{\|u_0\|_{L^{\infty}\left(\Omega\right)}\right\}\times (0,T)$ are measure zero sets in $\mathbb{R}^{d+2}$.\\
\begin{theorem}\label{Degenerate.Compactness.theorem}
{\rm The quasilinear viscous approximations $\left(u^{\varepsilon}\right)$ is $L^{1}\left(\Omega_{T}\right)$ compact.}	
\end{theorem}	
\begin{proof}
Following the proof of Theorem 4.5 of \cite{Ramesh_Mondal}, we conclude that $h=\widetilde{\chi^{\varepsilon}}(x,c;t)$ satisfies \eqref{Velocity.AveragingEquation1.A1}.  An application of Theorem \ref{VelocityAveragingLemma.12} with $h=\chi_{\widetilde{u^{\varepsilon}}}(c)$, $g=\widetilde{g^{\varepsilon}}$ and $\psi(c)=\chi_{[-\|u_{0}\|_{\infty},\|u_{0}\|_{\infty}]}(c)$, we conclude the proof of Theorem \ref{Degenerate.Compactness.theorem} as $\mbox{Vol}\left(\Omega_{T}\right)<\infty$ $\blacksquare$	
\end{proof}
\section{Entropy Solution as a Limit}\label{Degenerate.Section.3}
In this section, we show that the {\it a.e.} limit of quasilinear viscous approximations is the unique entropy solution
for scalar conservation laws in the sense of Otto \cite{MR1387428}. We now give some basic definitions before introducing the concepts of entropy solution. 
\begin{definition}[\textbf{Entropy}]\label{Otto.entropy}
	{\rm Let $\eta$ be a convex function. Then $\eta$ is said to be an entropy if for $j\in\left\{1,2,\cdots,d\right\}$, there exist function 
		$q_{j}$ such that for all $z\in\mathbb{R}$, the following equality 
		$$\eta^{\prime}(z)\,f_{j}^{\prime}(z)= q_{j}^{\prime}(z)$$
		hold.}
\end{definition}
Denote $Q:=\left(Q_{1},Q_{2},\cdots,Q_{d}\right).$
\begin{definition}[\textbf{Boundary-Entropy}]\cite[p.103]{Necas}\label{Otto.Bentropy}
	{\rm Let $H, Q\in C^{2}(\mathbb{R}^{2})$. A pair $\left(H, Q\right)$ is called a boundary entropy-entropy flux pair if for 
		$w\in\mathbb{R}$, $\left(H(.,w), Q(.,w)\right)$ is an entropy-entropy flux pair in the sense of Definition \ref{Otto.entropy}
		and $H, Q$ satisfy
		$$ H(w,w)=0,\,\, Q(w,w)=0,\,\,\partial_{1}\,H(w,w)=0,$$
		where $\partial_{1}H$ denotes the partial derivative with respect to the first variable.}
\end{definition}
\begin{remark}\cite[p.104]{Necas}\label{Boundary.Entropy.F.Otto.rmk1}
	Let $k\in\mathbb{R}$ be arbitary but fixed and let $l\in\mathbb{N}$. Define the entropy-entropy flux pair $\left(\eta_{l},q_{l}\right)$ by
	\begin{eqnarray}\nonumber\\
	\eta_{l}(z)&:=&\left(\left(z-k\right)^{2}+\left(\frac{1}{l}\right)^2\right)^{\frac{1}{2}}-\frac{1}{l},\nonumber\\
	q_{l}(z)&:=& \int_{k}^{z}\,\eta_{l}^{\prime}(r)\,f^{\prime}(r)\,dr.\nonumber
	\end{eqnarray}
	Then $\left(\eta_{l},q_{l}\right)$ converges uniformly to a non-smooth entropy-entropy flux pair $\left(|z-k|,\,F(z,k)\right)$ as $l\to\infty$, where
	$$F(z,k)=\mbox{sg}\left(z-k\right)\,\,\left(f(z)-f(k)\right).$$
	Denote the closed interval between real numbers $a,b$ by 
	$$\mathcal{I}[a,b]:=\left[\mbox{min}(a,b),\mbox{max}(a,b)\right].$$
	Define a boundary entropy-entropy flux pair $\left(H_{l},Q_{l}\right)$ by
	\begin{eqnarray}\nonumber
	H_{l}(z,w)&:=&\left(\left(\mbox{dist}\left(z,\mathcal{I}(w,k)\right)\right)^{2}+\left(\frac{1}{l}\right)^2\right)^{\frac{1}{2}}-\frac{1}{l},\nonumber\\
	Q_{l}(z,w)&:=& \int_{w}^{z}\,\partial_{1}H_{l}(r,w)\,f^{\prime}(r)\,dr.\nonumber
	\end{eqnarray}
	The sequence $\left(H_{l},Q_{l}\right)$ converges uniformly to $\left(\mbox{dist}(z,\mathcal{I}[w,k]),\mathcal{F}(z,w,k)\right)$ as $l\to\infty$,
	where $\mathcal{F}\in \left(C(\mathbb{R}^{3})\right)^{d}$ is given by 
	\[
	\mathcal{F}(z,w,k):=
	\begin{cases}
	f(w)-f(k) & \text{if $z\leq w\leq k$}\\
	0         & \text{if $w\leq z\leq k$}\\
	f(z)-f(k) & \text{if $w\leq k\leq z$}\\
	\vspace{0.1cm}\\
	f(k)-f(z) & \text{if $z\leq k\leq w$}\\
	0         & \text{if $k\leq z\leq w$}\\
	f(z)-f(w) & \text{if $k\leq w\leq z$}.\\
	\end{cases}
	\]
\end{remark}
\begin{definition}[\textbf{F. Otto}]\cite[p.103]{Necas}\label{F.Otto}
	Let $u_{0}\in L^{\infty}(\Omega)$ and $f\in\left(C^{1}(\mathbb{R})\right)^{d}$. We say that $u$ is an entropy solution for IBVP
	\eqref{ivp.cl} if and only if $u\in L^{\infty}(\Omega_{T})$ and satisfies:
	\begin{enumerate}
		\item  the conservation laws and the entropy condition in the sense
		\begin{eqnarray}\label{definition.otto.eqn1}
		\int_{\Omega_{T}}\left(\eta(u)\,\frac{\partial\phi}{\partial t}\,+\,\displaystyle\sum_{j=1}^{d}q_{i}(u)\,\frac{\partial\phi}
		{\partial x_{j}}\right)\,dx\,dt\geq 0
		\end{eqnarray}
		for all $\phi\in\mathcal{D}(\Omega_{T}),\,\,\,\phi\geq 0$ and all entropy-entropy flux pairs $\left(\eta, q\right)$;
		\item the boundary condition in the sense 
		\begin{eqnarray}\label{definition.otto.eqn2}
		\mbox{}ess\displaystyle\lim_{h\to 0}\int_{0}^{T}\int_{\partial\Omega}\,Q\left(u(r+h\nu(r)),0\right)\cdot\nu(r)\,\beta(r)
		\,\,\,dr\geq 0,
		\end{eqnarray}
		for all $\beta\in L^{1}(\partial\Omega\times(0,T)),\,\,\,\beta\geq 0$ {\it a.e.} and all boundary entropy fluxes $Q$;
		\item the initial condition $u_{0}\in L^{\infty}(\Omega)$ in the sense
		\begin{eqnarray}\label{definition.otto.eqn3}
		\mbox{ess}\displaystyle\lim_{t\to 0}\int_{\Omega}\left|u(x,t)-u_{0}(x)\right|\,\,dx=0.
		\end{eqnarray}
	\end{enumerate}
\end{definition}
\textbf{Proof of Theorem \ref{Deg.paper2.compensatedcompactness.theorem1}:} We prove Theorem \ref{Deg.paper2.compensatedcompactness.theorem1} in four steps for any dimension $d\in\mathbb{N}$. In Step 1, we show that the {\it a.e.} limit of a  subsequence of solutions to generalized viscosity problem \eqref{De.regularized.IBVP} is a weak solution of scalar conservation law \eqref{ivp.cl}. In Step 2, a weak solution satifies the inequality \eqref{definition.otto.eqn1}. In Step 3, we prove \eqref{definition.otto.eqn2} and \eqref{definition.otto.eqn3}. In Step 4, we conclude uniqueness of entropy solutions of IBVP for conservation laws \eqref{ivp.cl}.\\
	\textbf{Step 1:} Let $\phi \in \mathcal{D}(\Omega\times [0, T))$. Multiplying the equation \eqref{regularized.IBVP.a} by $\phi$, integrating over $\Omega_{T}$ and using integrating by parts formula, we get
	\begin{equation}\label{Deg.eqnchap502}
	\begin{split}
	\int_{0}^{T}\,\int_{\Omega}u^{\varepsilon,\delta}\,\frac{\partial \phi}{\partial t}\,dx\,dt - \varepsilon\,\displaystyle\sum_{j=1}^{d}\int_{0}^{T}\,\int_{\Omega}\,B(u^{\varepsilon,\delta})\,\frac{\partial u^{\varepsilon,\delta}}{\partial x_{j}}\,\frac{\partial \phi}{\partial x_{j}}\,dx\,dt +\displaystyle\sum_{j=1}^{d}\int_{0}^{T}\,\int_{\Omega}\,f_{j}(u^{\varepsilon,\delta})\,\frac{\partial \phi}{\partial x_{j}}\,dx\,dt\\ =\int_{\Omega}u^{\varepsilon,\delta}(x,0)\,\phi(x,0)\,dx\,dt +\delta\,\varepsilon\,\displaystyle\sum_{j=1}^{d}\int_{0}^{T}\,\int_{\Omega}\,B(u^{\varepsilon,\delta})\,\frac{\partial u^{\varepsilon,\delta}}{\partial x_{j}}\,\frac{\partial \phi}{\partial x_{j}}\,dx\,dt.
	\end{split}
	\end{equation}
	Observe that 
	\begin{equation}\label{Deg.LimitPass.1}
	\left|\displaystyle\sum_{j=1}^{d}\int_{0}^{T}\,\int_{\Omega}\,B(u^{\varepsilon,\delta})\,\frac{\partial u^{\varepsilon,\delta}}{\partial x_{j}}\,\frac{\partial \phi}{\partial x_{j}}\,dx\,dt\right|\leq \left\|\sqrt{B(u^{\varepsilon,\delta})}\,\frac{\partial u^{\varepsilon,\delta}}{\partial x_{j}}\right\|_{L^{2}\left(\Omega_{T}\right)}\,\left\|\,\frac{\partial\phi}{\partial x_{j}}\right\|_{L^{2}\left(\Omega_{T}\right)}
	\end{equation}
	Using \eqref{Deg.LimitPass.1},\,Theorem \ref{regularized.chap3thm1}, Theorem \ref{Deg.Compactness.lemma.1}, Lemma \ref{De.hypothesisDinitialdatalem} and Dominated convergence theorem, we pass to the limit as $\delta\to 0$ in \eqref{Deg.LimitPass.1} to obtain 
	\begin{equation}\label{eqnchap502}
	\begin{split}
	\int_{0}^{T}\,\int_{\Omega}u^{\varepsilon}\,\frac{\partial \phi}{\partial t}\,dx\,dt - \varepsilon\,\displaystyle\lim_{\delta\to 0}\left(\displaystyle\sum_{j=1}^{d}\int_{0}^{T}\,\int_{\Omega}\,B(u^{\varepsilon,\delta})\,\frac{\partial u^{\varepsilon,\delta}}{\partial x_{j}}\,\frac{\partial \phi}{\partial x_{j}}\,dx\,dt\right) +\displaystyle\sum_{j=1}^{d}\int_{0}^{T}\,\int_{\Omega}\,f_{j}(u^{\varepsilon})\,\frac{\partial \phi}{\partial x_{j}}\,dx\,dt\\ =\int_{\Omega}u_{0\varepsilon}(x,0)\,\phi(x,0)\,dx\,dt.
	\end{split}
	\end{equation}
	Applying Lemma \ref{Degenerate.Maximum.Principle} and Dominated Convergence Theorem, we pass to the limit as $\varepsilon\to 0$ to conclude  
	\begin{equation}\label{weaksolution.eqn1}
	\displaystyle\lim_{\varepsilon\to 0}\int_{0}^{T}\,\int_{\Omega}u^{\varepsilon}\,\frac{\partial \phi}{\partial t}\,dx\,dt =\int_{0}^{T}\,\int_{\Omega}u\,\frac{\partial \phi}{\partial t}\,dx\,dt,
	\end{equation}
	For $j=1,2,\cdots,d$,  proofs of 
	\begin{equation}\label{weaksolution.eqn2}
	\displaystyle\lim_{\varepsilon\to 0}\int_{0}^{T}\,\int_{\Omega}\,f_{j}(u^{\varepsilon})\,\frac{\partial \phi}{\partial x_{j}}\,dx\,dt = \int_{0}^{T}\,\int_{\Omega}\, f_{j}(u)\,\frac{\partial\phi}{\partial x_{j}}\,dx\,dt.
	\end{equation}
Applying Lemma \ref{regularized.chap3thm1} and Dominated convergence theorem, we conclude	
	\begin{equation}\label{weaksolution.eqn4}
	\int_{\Omega}u_{0\varepsilon}(x)\,\phi(x,0)\,dx\to \int_{\Omega} u(x,0)\,\phi(x,0)\,dx.
	\end{equation}
	Finally, for $j=1,2,\cdots,d$, we show that
	\begin{equation}\label{eqnchap504}
	\displaystyle\lim_{\varepsilon\to 0}\varepsilon\left(\displaystyle\lim_{\delta\to 0}\,\left(\int_{0}^{T}\,\int_{\Omega}\,B(u^{\varepsilon,\delta})\,\frac{\partial u^{\varepsilon,\delta}}{\partial x_{j}}\,\frac{\partial \phi}{\partial x_{j}}\,dx\,dt\right)\right) = 0.
	\end{equation}
	Since $\varepsilon\left(\displaystyle\lim_{\delta\to 0}\,\left(\int_{0}^{T}\,\int_{\Omega}\,B(u^{\varepsilon,\delta})\,\frac{\partial u^{\varepsilon,\delta}}{\partial x_{j}}\,\frac{\partial \phi}{\partial x_{j}}\,dx\,dt\right)\right)$ exists, therefore we have
	\begin{equation}\label{Deg.LimitPass.2}
	\displaystyle\lim_{\delta\to 0}\left|\varepsilon\left(\,\left(\int_{0}^{T}\,\int_{\Omega}\,B(u^{\varepsilon,\delta})\,\frac{\partial u^{\varepsilon,\delta}}{\partial x_{j}}\,\frac{\partial \phi}{\partial x_{j}}\,dx\,dt\right)\right)\right|=\left|\varepsilon\left(\displaystyle\lim_{\delta\to 0}\,\left(\int_{0}^{T}\,\int_{\Omega}\,B(u^{\varepsilon,\delta})\,\frac{\partial u^{\varepsilon,\delta}}{\partial x_{j}}\,\frac{\partial \phi}{\partial x_{j}}\,dx\,dt\right)\right)\right|.
	\end{equation}
	Appying \eqref{Deg.LimitPass.1} and Theorem \ref{Deg.Compactness.lemma.1}, we get
	\begin{equation}\label{Deg.LimitPass.3}
	\left|\varepsilon\left(\displaystyle\lim_{\delta\to 0}\,\left(\int_{0}^{T}\,\int_{\Omega}\,B(u^{\varepsilon,\delta})\,\frac{\partial u^{\varepsilon,\delta}}{\partial x_{j}}\,\frac{\partial \phi}{\partial x_{j}}\,dx\,dt\right)\right)\right|\leq \sqrt{\varepsilon}\,A\,\sqrt{\frac{\mbox{Vol}\left(\Omega\right)}{2}}\,\left\|\,\frac{\partial\phi}{\partial x_{j}}\right\|_{L^{2}\left(\Omega_{T}\right)}.	
	\end{equation}	
	Therefore we conclude \eqref{eqnchap504}.
	Equations \eqref{weaksolution.eqn1},\eqref{weaksolution.eqn2},\eqref{weaksolution.eqn4},\eqref{eqnchap504} together give 
	\begin{equation}\label{eqnchap512}
	\int_{0}^{T}\,\int_{\Omega}u\,\frac{\partial \phi}{\partial t}\,dx\,dt +\displaystyle\sum_{j=1}^{d}\int_{0}^{T}\,\int_{\Omega} f_{j}(u)\,\frac{\partial \phi}{\partial x_{j}}\,dx \,dt =\int_{\Omega}u_{0}(x)\,\phi(x,0)\,dx\,dt.
	\end{equation}
	Therefore $u$ is a weak solution of initial value problem for conservation law \eqref{ivp.cl}.\\
\textbf{Step 2:}
For $k\in\mathbb{R}$, Kruzhkov's entropy-entropy flux pairs are given by
\begin{eqnarray}\label{chap70eqn2}
\eta(u) &=& |u-k|,\nonumber\\
q(u) &=& \mbox{sg}(u-k)\,\,\left(f(u)-f(k)\right),
\end{eqnarray}
where for $j=1,2,\cdots,d$, $\eta$ and $q =\left(q_{1},q_{2},\cdots,q_{d}\right)$ satisfy the following relation
$$\eta^{'}(u)f_{j}^{'}(u)= q_{j}^{'}(u).$$
Multiplying \eqref{regularized.IBVP.a} by $\eta^{'}(u^{\varepsilon,\delta})$ and applying chain rule, we get
\begin{equation}\label{chap70eqn4}
\frac{\partial \eta(u^{\varepsilon,\delta})}{\partial t} + \displaystyle\sum_{j=1}^{d}\frac{\partial q_{j}(u^{\varepsilon,\delta})}{\partial x_{j}}  = \varepsilon\,\displaystyle\sum_{j=1}^{d} \eta^{'}(u^{\varepsilon,\delta}) \frac{\partial }{\partial x_{j}}\left(B(u^{\varepsilon,\delta})\frac{\partial u^{\varepsilon,\delta}}{\partial x_{j}}\right). 
\end{equation}
Let $\phi\in \mathcal{D}(\Omega_{T})$ and $\phi\geq 0$. Multiplying \eqref{chap70eqn4} by $\phi$, integrating over $\Omega_{T}$ and using integration by parts formula, we obtain
\begin{equation}\label{Deg.chap70eqn5}
-\int_{\Omega_{T}}\eta(u^{\varepsilon,\delta})\frac{\partial\phi}{\partial t}\,dx dt -\displaystyle\sum_{j=1}^{d}\int_{\Omega_{T}} q_{j}(u^{\varepsilon,\delta})\frac{\partial\phi}{\partial x_{j}}\,dx\,dt = \varepsilon\,\displaystyle\sum_{j=1}^{d}\int_{\Omega_{T}}\eta^{'}(u^{\varepsilon,\delta})\phi \frac{\partial }{\partial x_{j}}\left(B(u^{\varepsilon,\delta})\frac{\partial u^{\varepsilon,\delta}}{\partial x_{j}}\right)\,dx dt.
\end{equation}
Passing to the limit as $\delta\to 0$ in \eqref{Deg.chap70eqn5}, we get
\begin{equation}\label{chap70eqn5}
\begin{split}
-\int_{\Omega_{T}}\eta(u^{\varepsilon})\frac{\partial\phi}{\partial t}\,dx dt -\displaystyle\sum_{j=1}^{d}\int_{\Omega_{T}} q_{j}(u^{\varepsilon})\frac{\partial\phi}{\partial x_{j}}\,dx\,dt\hspace{5cm}\\ = \varepsilon\,\displaystyle\lim_{\delta\to 0}\left(\displaystyle\sum_{j=1}^{d}\int_{\Omega_{T}}\eta^{'}(u^{\varepsilon,\delta})\phi \frac{\partial }{\partial x_{j}}\left(B(u^{\varepsilon,\delta})\frac{\partial u^{\varepsilon,\delta}}{\partial x_{j}}\right)\,dx dt\right).
\end{split}
\end{equation}
Taking $\limsup$ as $\varepsilon\to 0$ on both sides of \eqref{chap70eqn5}, we have
\begin{eqnarray}\label{chap70eqn6}
-\int_{\Omega_{T}}\eta(u)\frac{\partial\phi}{\partial t}\,dx dt -\displaystyle\sum_{j=1}^{d}\int_{\Omega_{T}} q_{j}(u)\frac{\partial\phi}{\partial x_{j}}\,dx dt\hspace{5cm}\nonumber\\ = \displaystyle\limsup_{\varepsilon\to 0}\left(\displaystyle\limsup_{\delta\to 0}\left(\displaystyle\sum_{j=1}^{d}\left(\varepsilon\int_{\Omega_{T}} \eta^{'}(u^{\varepsilon,\delta})\phi \frac{\partial }{\partial x_{j}}\left(B(u^{\varepsilon,\delta})\frac{\partial u^{\varepsilon,\delta}}{\partial x_{j}}\right)\,dx dt\right)\right)\right).
\end{eqnarray}
Let $G:\mathbb{R}\to\mathbb{R}$ be a $C^{\infty}$ function such that 
$$G(x) = |x|\,\,\mbox{for} |x|\geq 1\,\,\mbox{and}\,\,G^{''}\geq 0.$$
For $\gamma > 0$, define $G_{\gamma}:\mathbb{R}\to\mathbb{R}$ by 
$$G_{\gamma}(y) = \gamma G\left(\frac{(y-k)}{\gamma}\right)$$
so that $G_{\gamma}(y)\to |y-k|$ as $\gamma\to 0$, see\cite[p.72]{MR1304494}.  The function $G_{\gamma}$ is a convex function and $G^{'}_{\gamma}(v)\to \mbox{sg}(v-k)$ as $\gamma\to 0$. We can use the antiderivative of the sequence of approximations of sign function used in \cite{MR542510} as a choice for $G_{\gamma}$ for our computation as we do not need $C^{\infty}-$ function.\\
For $j=1,2,\cdots,d$, note that
\begin{eqnarray}\label{chap70eqn7}
\varepsilon \int_{\Omega_{T}} \eta^{\prime}(u^{\varepsilon,\delta})\phi \frac{\partial }{\partial x_{j}}\left(B(u^{\varepsilon,\delta})\frac{\partial u^{\varepsilon,\delta}}{\partial x_{j}}\right)\,dx dt &=&\varepsilon\displaystyle\lim_{\gamma\to 0}\int_{\Omega_{T}} G_{\gamma}^{'}(u^{\varepsilon,\delta})\phi \frac{\partial }{\partial x_{j}}\left(B(u^{\varepsilon,\delta})\frac{\partial u^{\varepsilon}}{\partial x_{j}}\right)\,dx dt.\nonumber\\
&=&\varepsilon\displaystyle\limsup_{\gamma\to 0}\int_{\Omega_{T}} G_{\gamma}^{'}(u^{\varepsilon,\delta})\phi \frac{\partial }{\partial x_{j}}\left(B(u^{\varepsilon,\delta})\frac{\partial u^{\varepsilon}}{\partial x_{j}}\right)\,dx dt.\nonumber\\
&\leq& -\varepsilon \displaystyle\displaystyle\liminf_{\gamma\to 0}\int_{\Omega_{T}} G_{\gamma}^{''}(u^{\varepsilon,\delta})\phi B(u^{\epsilon,\delta})\left( \frac{\partial u^{\varepsilon,\delta}}{\partial x_{j}}\right)^{2}\,dx dt\nonumber\\
&-&\varepsilon \displaystyle\liminf_{\gamma\to 0}\int_{\Omega_{T}} G_{\gamma}^{'}(u^{\varepsilon,\delta})\frac{\partial\phi}{\partial x_{j}} B(u^{\varepsilon,\delta})\frac{\partial u^{\varepsilon,\delta}}{\partial x_{j}}\, dx dt.\nonumber\\
&=& -\varepsilon \displaystyle\displaystyle\liminf_{\gamma\to 0}\int_{\Omega_{T}} G_{\gamma}^{''}(u^{\varepsilon,\delta})\phi B(u^{\epsilon,\delta})\left( \frac{\partial u^{\varepsilon,\delta}}{\partial x_{j}}\right)^{2}\,dx dt\nonumber\\
&-&\varepsilon \displaystyle\lim_{\gamma\to 0}\int_{\Omega_{T}} G_{\gamma}^{'}(u^{\varepsilon,\delta})\frac{\partial\phi}{\partial x_{j}} B(u^{\varepsilon,\delta})\frac{\partial u^{\varepsilon,\delta}}{\partial x_{j}}\, dx dt.
\end{eqnarray}
Therefore, we have
\begin{eqnarray}\label{Deg.Entropy.Equation10}
\varepsilon \int_{\Omega_{T}} \eta^{\prime}(u^{\varepsilon,\delta})\phi \frac{\partial }{\partial x_{j}}\left(B(u^{\varepsilon,\delta})\frac{\partial u^{\varepsilon,\delta}}{\partial x_{j}}\right)\,dx dt&\leq& -\varepsilon \displaystyle\displaystyle\liminf_{\gamma\to 0}\int_{\Omega_{T}} G_{\gamma}^{''}(u^{\varepsilon,\delta})\phi B(u^{\epsilon,\delta})\left( \frac{\partial u^{\varepsilon,\delta}}{\partial x_{j}}\right)^{2}\,dx dt\nonumber\\
&+&\left|\varepsilon \displaystyle\lim_{\gamma\to 0}\int_{\Omega_{T}} G_{\gamma}^{'}(u^{\varepsilon,\delta})\frac{\partial\phi}{\partial x_{j}} B(u^{\varepsilon,\delta})\frac{\partial u^{\varepsilon,\delta}}{\partial x_{j}}\, dx dt\right|.
\end{eqnarray}
Since $\left(\sqrt{\varepsilon}\|\sqrt{B\left(u^{\varepsilon,\delta}\right)}\frac{\partial u^{\varepsilon,\delta}}{\partial x_{j}}\|_{L^{2}(\Omega_{T})}\right)_{\varepsilon,\delta > 0}$ is bounded, for $j=1,2,\cdots,d$ and for all $\gamma>0$, there exists $M>0$ such that $\left|G_{\gamma}^{\prime}\left(u^{\varepsilon,\delta}\right)\right|\leq M$. \\
Since\,\,$\displaystyle\lim_{\gamma\to 0}\int_{\Omega_{T}} G_{\gamma}^{'}(u^{\varepsilon,\delta})\frac{\partial\phi}{\partial x_{j}} B(u^{\varepsilon,\delta})\frac{\partial u^{\varepsilon,\delta}}{\partial x_{j}}\, dx dt$\,\, exists, we get
$$\displaystyle\lim_{\gamma\to 0}\Big|\varepsilon\int_{\Omega_{T}} G_{\gamma}^{'}(u^{\varepsilon,\delta})\frac{\partial\phi}{\partial x_{j}} B(u^{\varepsilon,\delta})\frac{\partial u^{\varepsilon,\delta}}{\partial x_{j}}\, dx dt.\Big|=\Big|\varepsilon \displaystyle\lim_{\gamma\to 0}\int_{\Omega_{T}} G_{\gamma}^{'}(u^{\varepsilon,\delta})\frac{\partial\phi}{\partial x_{j}} B(u^{\varepsilon,\delta})\frac{\partial u^{\varepsilon,\delta}}{\partial x_{j}}\, dx dt\Big|.$$
Observe that
\begin{equation}\label{Deg.Entropy.Equation1}
\Big|\varepsilon\int_{\Omega_{T}} G_{\gamma}^{'}(u^{\varepsilon,\delta})\frac{\partial\phi}{\partial x_{j}} B(u^{\varepsilon,\delta})\frac{\partial u^{\varepsilon,\delta}}{\partial x_{j}}\, dx dt.\Big|\leq M\sqrt{\varepsilon}\left(\sqrt{\varepsilon}\left\|\sqrt{B\left(u^{\varepsilon,\delta}\right)}\frac{\partial u^{\varepsilon,\delta}}{\partial x_{j}}\right\|_{L^{2}(\Omega_{T})}\right)\,\left\|\frac{\partial\phi}{\partial x_{j}}\right\|_{L^{2}(\Omega_{T})}
\end{equation} 
Applying Theorem \ref{Deg.Compactness.lemma.1}, we get
\begin{equation}\label{Deg.Entropy.Equation2}
\Big|\varepsilon\int_{\Omega_{T}} G_{\gamma}^{'}(u^{\varepsilon,\delta})\frac{\partial\phi}{\partial x_{j}} B(u^{\varepsilon,\delta})\frac{\partial u^{\varepsilon,\delta}}{\partial x_{j}}\, dx dt.\Big|\leq M\sqrt{\varepsilon}\,A\,\sqrt{\frac{\mbox{Vol}\left(\Omega\right)}{2}}\,\left\|\frac{\partial\phi}{\partial x_{j}}\right\|_{L^{2}(\Omega_{T})}
\end{equation} 
We can easily deduce that
\begin{equation*}
\displaystyle\limsup_{\delta\to 0}\Big|\varepsilon \displaystyle\lim_{\gamma\to 0}\int_{\Omega_{T}} G_{\gamma}^{'}(u^{\varepsilon,\delta})\frac{\partial\phi}{\partial x_{j}} B(u^{\varepsilon,\delta})\frac{\partial u^{\varepsilon,\delta}}{\partial x_{j}}\, dx dt.\Big|\to 0\,\,\mbox{as}\,\,\varepsilon\to 0.
\end{equation*}
Hence 
\begin{equation}\label{Entropy.Otto.Deg.eqn3}
\displaystyle\limsup_{\varepsilon\to 0}\left(\displaystyle\limsup_{\delta\to 0}\Big|\varepsilon \displaystyle\lim_{\gamma\to 0}\int_{\Omega_{T}} G_{\gamma}^{'}(u^{\varepsilon,\delta})\frac{\partial\phi}{\partial x_{j}} B(u^{\varepsilon,\delta})\frac{\partial u^{\varepsilon,\delta}}{\partial x_{j}}\, dx dt.\Big|\right)= 0.
\end{equation}
Observe that
\begin{equation}\label{entropy.solution.eqn2}
-\displaystyle\liminf_{\varepsilon\to 0}\left(\displaystyle\liminf_{\delta\to 0}\left(\varepsilon\displaystyle\sum_{j=1}^{d} \displaystyle\lim_{\gamma\to 0}\int_{\Omega_{T}} G_{\gamma}^{''}(u^{\varepsilon,\delta})\phi B(u^{\epsilon,\delta})\left( \frac{\partial u^{\varepsilon,\delta}}{\partial x_{j}}\right)^{2}\,dx dt\right)\right)\leq 0.
\end{equation}
Using \eqref{Entropy.Otto.Deg.eqn3}, \eqref{entropy.solution.eqn2} in \eqref{chap70eqn6} together with \eqref{Deg.Entropy.Equation10}, we have required inequality \eqref{definition.otto.eqn1}.\\
\textbf{Step 3:} The proofs of \eqref{definition.otto.eqn2} and \eqref{definition.otto.eqn2} follow from the proof of Theorem 1.1 of \cite{Ramesh_Mondal}.
\\
\textbf{Step 4:} The proof of uniqueness of entropy solution follows from 
\cite[p.113]{Necas} $\blacksquare$
\newpage
\bibliographystyle{plain}

\end{document}